\documentclass[reqno,11pt]{amsart}
\usepackage{amssymb}
\usepackage{verbatim}
\usepackage[all,cmtip]{xy}
\newif\ifpdf
\ifpdf
  \usepackage[pdftex]{graphicx}
  \usepackage[pdftex]{hyperref}
\else
  \usepackage{graphicx}
\fi

\numberwithin{equation}{section}       
 \theoremstyle{plain}    
 \newtheorem{thm}{Theorem}[section]
 \numberwithin{equation}{section} 
 \numberwithin{figure}{section} 
 \theoremstyle{plain}
 \theoremstyle{plain}    
 \theoremstyle{plain}    
 \newtheorem{pro}[thm]{Proposition} 
 \theoremstyle{plain}    
 \newtheorem{lem}[thm]{Lemma} 
 \theoremstyle{remark}
 \newtheorem{rem}[thm]{Remark}
 \theoremstyle{definition}

\newtheorem*{thmM}{Main Theorem}

\newtheorem*{properties}{Properties}

\theoremstyle{definition}
\newtheorem{defi}[thm]{Definition}

\newtheorem*{ackn}{Acknowledgement}

\newcommand{\Q}{{\mathbb{Q}}}
\newcommand{\R}{{\mathbb{R}}}

\newcommand{\cE}{{\mathcal{E}}}

\renewcommand{\a}{\alpha}

\newcommand{\f}{\varphi}

\newcommand{\MA}{\mathrm{MA}\,}
\newcommand{\Amp}{\mathrm{Amp}\,}

\newcommand{\vol}{\operatorname{vol}}

\newcommand{\supp}{\operatorname{supp}}

\newcommand{\ay}{\alpha_{Y}}

\newcommand{\te}{\theta}
\begin{document}

\setcounter{tocdepth}{1}

\title[Divisorial Zariski Decomposition and full mass currents]{Divisorial Zariski Decomposition and some properties of full mass currents}

\date{\today}

\author{E.~Di Nezza}
\address{Department of Mathematics, Imperial College London, SW7 2AZ, UK}
\email{e.di-nezza@imperial.ac.uk}
\author{E.~Floris}
\address{Department of Mathematics, Imperial College London, SW7 2AZ, UK}
\email{e.floris@imperial.ac.uk}
\author{S.~Trapani}
\address{Dipartimento di Matematica, Universit\`a di Roma `Tor Vergata ', Rome, Italy}
\email{trapani@axp.mat.uniroma2.it}

\begin{abstract}
Let $\alpha$ be a big class on a compact K\"ahler manifold. We prove that a decomposition $\alpha=\alpha_1+\alpha_2$ into the sum of a modified nef class $\alpha_1$ and a pseudoeffective class $\alpha_2$ is the divisorial Zariski decomposition of $\alpha$ if and only if $\vol(\alpha)=\vol(\alpha_1)$. We deduce from this result some properties of full mass currents.
\end{abstract} 

\maketitle

\tableofcontents

\section*{Introduction}
The study of the Zariski decomposition started with the work of Zariski \cite{Zar} who defined it for an effective divisor in a smooth projective surface. Fujita extended the definition to the case of pseudoeffective divisors \cite{F}. 
Due to the importance of the Zariski decomposition for surfaces,
several generalizations to higher dimension exist (see \cite{ProkZar} for a survey of these constructions).
The divisorial Zariski decomposition for a cohomology class $\alpha$ on a K\"ahler manifold has been introduced by Boucksom in \cite{Bou2}. If $\alpha$ is the class of a divisor on a projective manifold, the divisorial Zariski decomposition coincides with the $\sigma$-decomposition introduced by Nakayama \cite{N}.
The divisorial Zariski decomposition is a decomposition $$\alpha=Z(\alpha)+\{N(\alpha)\}$$
into a ``positive part", the Zariski projection $Z(\alpha)$,
whose non-nef locus has codimension at least $2$,
and a ``negative part" $\{N(\alpha)\}$
which is the class of an effective divisor and is rigid.
The class $Z(\alpha)$ encodes some important information about $\alpha$:
$Z(\alpha)$ is big if and only if $\alpha$ is and $\vol(\alpha)=\vol(Z(\alpha))$.

In this note we give a criterion for a sum of two classes to be a divisorial Zariski decomposition.
Our main result is:
\begin{thmM}
Let $X$ be a compact K\"ahler manifold of complex dimension $n$. Let $\alpha$ be a big class on $X$. Let $\alpha_1\in H^{1,1}(X,\R)$ be a modified nef class and $\alpha_2\in H^{1,1}(X,\R)$ be a pseudoeffective class. Then $\alpha=\alpha_1+\alpha_2$ is the divisorial Zariski decomposition of $\alpha$ if and only if $\vol(\alpha)=\vol(\alpha_1)$.
\end{thmM}
\noindent The relations between the Zariski decomposition of numerical classes of cycles on a projective variety and their volume have been largely studied recently in a series of papers \cite{L}, \cite{FL}, \cite{FKL}. 
The Main Theorem also goes in this direction: for instance, if $X$ is projective and $\alpha=\{D\}$ is the class of a big divisor, we recover \cite[Proposition 5.3]{FL} for cycles of codimension $1$.

Our proof relies deeply on a result of existence and uniqueness of weak solutions of complex Monge-Amp\`ere equations.\\
On the other hand the proof in \cite{FL} uses the differentiability of the volume function  $f(t)=\vol(\alpha+t\{D\})$, known at the moment to be true only in the algebraic case. In Remark \ref{zariskiproj} we present a proof
of the Main Theorem using the differentiability of the volume. 
As it is proved by Xiao \cite[Proposition 1.1]{X} , the differentiability of the volume is equivalent to the following quantitative version of a Demailly's conjecture \cite[Conjecture 10.1]{BDPP}, that states as follows:\\
\emph{Let $X$ be a compact K\"ahler manifold of complex dimension $n$, and let $\alpha, \beta\in H^{1,1}(X, \R)$ be two nef classes. Then we have}
\begin{equation}\label{conj}
\vol(\alpha-\beta)\geq \alpha^n - n \,\alpha^{n-1}\cdot \beta.
\end{equation}
\vspace{2mm}

We then show that the Main Theorem is strictly related to the invariance of finite energy classes under bimeromorphic maps. More precisely, in Theorem \ref{ConditionV} we show that finite energy classes are inviariant under a bimeromorphic map if and only if the volumes are preserved. This extends to any dimension \cite[Proposition 2.5]{dn} where a similar statement is proved in dimension $2$ by the first named author using the Hodge index theorem.
\vspace{2mm}

We now give a brief outline of this note. Section $1$ reviews background material on the divisorial Zariski decomposition and currents with full Monge-Amp\`ere mass. In Section $2$ we prove the Main Theorem and in Section $3$ we give some applications to full mass currents. In particular we prove Theorem \ref{ConditionV}. 
\begin{ackn}
We would like to thank S\'ebastien Boucksom for several useful discussions on the subject and for communicating us the proof in Remark \ref{zariskiproj}.
\end{ackn}

\section{Preliminaries}
\noindent Let $(X,\omega)$ be a compact K{\"a}hler manifold of complex dimension $n$ and let $\a\in H^{1,1}(X,\R)$ be a real $(1,1)$-cohomology class.
Recall that $\a$ is said to be \emph{pseudo-effective} if it can be represented by a closed positive $(1,1)$-current $T$; $\alpha$ is  \emph{nef} if and only if, for any $\varepsilon>0$ there exists a smooth form $\theta_\varepsilon\in \alpha$ such that $\theta_\varepsilon\geq -\varepsilon\omega$; $\alpha$ is \emph{big} if and only if it can be represented by a \emph{K{\"a}hler current}, i.e. if and only if there exists a positive closed $(1,1)$-current $T\in\a $ such that $T\geq \varepsilon\,\omega$ for some $\varepsilon>0$ and $\alpha$ is a K\"ahler class if and only if it contains a K\"ahler form.

\vspace{2mm}
Given a smooth representative $\te$ of the class $\a$, it follows from $\partial\bar{\partial}$-lemma that any positive $(1,1)$-current can be written as $T=\te+dd^c \f$ where the global potential $\f$ is a \emph{$\te$-psh function}, i.e. $\te+dd^c\f\geq 0$. Here, $d$ and $d^c$ are real differential operators defined as
$$d:=\partial +\bar{\partial},\qquad d^c:=\frac{i}{2\pi}\left(\bar{\partial}-\partial \right).$$

\vspace{2mm}
Let $T$ be a closed positive $(1,1)$-current. We denote by $\nu(T,x)$ (resp. $\nu(T,D)$) its Lelong number at a point $x\in X$ (resp. along a prime divisor $D$). We refer the reader to \cite{Dem} for the definition.\\
There is a unique decomposition of $T$ as a weakly convergent series
$$T =R+ \sum_j \lambda_j [D_j]$$
where:
\begin{itemize}
\item[(i)] $[D_j]$ is the current of integration over the prime divisor $D_j\subset X$, 
\item[(ii)]$\lambda_j :=\nu(T,D_j)\geq 0$,
\item[(iii)]$R$ is a closed positive current with the property that ${\rm codim} E_c(R) \geq 2$ for every $c > 0$. 
\end{itemize}
Recall that 
$$E_c(R):=\{x\in X\;:\; \nu(R,x)\geq c\}$$ and that this is an analytic subset of $X$ by a famous result due to Siu \cite{Siu}.\\

\noindent Such a decomposition is called the Siu decomposition of $T$.

\subsubsection{Analytic and minimal singularities}
 
A positive current $T=\theta+dd^c\f$ is said to have \emph{analytic singularities} if there exists $c>0$ such that locally on $X$,
$$
\f=\frac{c}{2}\log\sum_{j=1}^{N}|f_j|^2+u,
$$
where $u$ is smooth and $f_1,\ldots,f_N$ are local holomorphic functions.
\vspace{2mm}

If $T$ and $T'$ are two closed positive currents on $X$, then $T'$ is said to be \emph{less singular} than $T$ if their local potentials satisfy $\f\le\f'+O(1)$.\\
A positive current $T$ is said to have \emph{minimal singularities} (inside its cohomology class  $\a$) if it is less singular than any other positive current in $\a$. Its $\theta$-psh potentials $\f$ will correspondingly be said to have minimal singularities.

Such $\theta$-psh functions with minimal singularities always exist, one can consider for example
$$V_\theta:=\sup\left\{ \f\,\,\theta\text{-psh}, \f\le 0\text{ on } X \right \}.$$

\subsection{Big and Modified nef classes}
\begin{defi}\label{ample}
{\it If $\a$ is a big class, we define its \emph{ample locus} $\Amp(\a)$ as the set of points $x\in X$ such that there exists a K\"ahler current $T\in\a$ with analytic singularities and smooth in a neighbourhood of $x$.}
\end{defi}

The ample locus $\Amp(\a)$ is a Zariski open subset, and it is nonempty thanks to Demailly's regularization result (see \cite{Bou2}).

\vspace{2mm}
Observe that a current with minimal singularities $T_{\min}\in \alpha$ has locally bounded potential in $\Amp(\alpha)$.
\begin{defi}
{\it Let $\alpha$ be a big class.
\begin{enumerate}
\item Let $T\in\alpha$ be a positive $(1,1)$-current, then we set $$E_+(T):= \{x\in X\;:\; \nu(T,x)>0\}.$$
\item We define the \emph{non K\"ahler locus} of $\alpha$ as $$E_{nk}(\a):= \bigcap_T E_+(T)$$ ranging among all the K\"ahler currents in $\alpha$.
\end{enumerate}
}
\end{defi}
\noindent 
By \cite[Theorem 3.17(iii)]{Bou2} a class $\alpha$ is K\"ahler if and only if $E_{nk}(\alpha)= \emptyset$. Moreover by \cite[Theorem 3.17(ii)]{Bou2} we have $E_{nk}(\alpha)=X\setminus \Amp(\alpha)$.

\begin{defi}
{\it We say that $\alpha$ is \emph{modified nef} if and only if for every $\varepsilon>0$ there exists a closed $(1,1)$-current $T_\varepsilon\in \alpha$ with $T_\varepsilon\geq -\varepsilon \omega$ and $\nu(T_\varepsilon,D)=0$ for any prime divisor $D$.
}
\end{defi}
We recall now an alternative and useful defintion of modified nef classes.
\begin{pro}\cite[Proposition 3.2]{Bou2}\label{bou2}
Let $\alpha\in H^{1,1}(X, \R)$ be a pseudoeffective class. Then $\alpha$ is modified nef if and only if $\nu(\alpha, D)=0$ for every prime divisor $D$.
\end{pro}
\noindent We refer to \cite{Bou2} for the defintion and properties of the \emph{minimal multiplicity} $\nu(\alpha,D)$.


\subsection{The Divisorial Zariski decomposition}
In this subsection we collect some basic results on the divisorial Zariski decomposition defined in \cite{Bou2}. They can all be found in \cite{Bou2} but we recall some statements widely used in this note.\\
Let $\alpha\in H^{1,1}(X, \R)$ be a pseudo-effective class. The \emph{divisorial Zariski decomposition} of $\alpha$ is defined as follows:
\begin{defi}
{\it  The negative part of $\alpha$ is defined as $N(\alpha):= \sum \nu(\alpha, D) [D]$, where $D$ are prime divisors. The Zariski projection of $\alpha$ is $Z(\alpha):= \alpha - \{N(\alpha)\}$. We call the decomposition $\alpha= Z(\alpha)+ \{N(\alpha)\}$ the divisorial Zariski decomposition of $\alpha$.
}
\end{defi}
\begin{properties}\label{properties}
Let $\alpha= Z(\alpha)+ \{N(\alpha)\}$ be the divisorial Zariski decomposition of $\alpha$. Then:
\begin{enumerate}
\item The class $Z(\alpha)$ is modified nef \cite[Proposition 3.8]{Bou2}. 
\item  $N(\alpha)$ is a divisor, i.e. there is a finite number of prime divisors $D$ such that $\nu(\alpha,D)>0$ \cite[Proposition 3.12]{Bou2}.
\item The set of modified nef classes is a closed convex cone and it is the closure of the convex cone generated by the classes $\mu_\star \tilde{\alpha}$ where $\mu: \tilde{X}\rightarrow X$ is a modification and $\tilde{\alpha}$ is a K\"ahler class on $\tilde{X}$ \cite[Proposition 2.3]{Bou2}. 
\item The negative part $\{N(\alpha)\}$ is a \emph{rigid} class, i.e. it contains only one positive current \cite[Proposition 3.13]{Bou2}.
\item Let $\alpha$ be a modified nef and big class, $D_1,\ldots, D_k$ be prime divisors and $\lambda_1,\ldots,\lambda_k\in \R^+$. Then \cite[Proposition 3.18]{Bou2}
$$N(\alpha+\sum_i \lambda_i \{D_i\})=\sum_i \lambda_i [D_i]$$
if and only if $D_j\subset E_{nk}(\alpha)$ for any $j$.
\end{enumerate}
\end{properties}

\begin{pro}\cite[Proposition 3.6(ii)]{Bou2}\label{bou6ii}
Let $\alpha\in H^{1,1}(X, \R)$ be a big class and let $T_{\min}\in \alpha$ be a current with minimal singularities. Consider the Siu decomposition of $T_{\min}$,
$$T_{\min}=R+\sum_j a_j [D_j]$$ where $a_j=\nu(T_{\min}, D_j)$. Then $\{R\}=Z(\alpha)$ and $\{\sum_j a_j D_j\}= \{N(\alpha)\}$. In particular, $\nu(\alpha, D)=\nu(T_{\min}, D)$ for any prime divisor $D$.
\end{pro}

\subsection{Volume of big classes.}
Fix $\a\in H_{big}^{1,1}(X,\R)$. We introduce

\begin{defi}
\it{Let $T_{\min}$ be a current with minimal singularities in $\a$  and let $\Omega$ a Zariski open set on which the potentials of $T_{\min}$ are locally bounded, then 
\begin{equation}\label{v}
\vol(\a):=\int_{\Omega} T_{\min}^n>0
\end{equation}
is called the volume of $\a$.}
\end{defi}
\indent Note that the Monge-Amp{\`e}re measure of $T_{\min}$
is well defined in $\Omega$ by \cite{bt} and that the volume is independent of the choice of $T_{\min}$ and $\Omega$ (\cite[Theorem 1.16]{begz}).
 
Let $f\colon X\rightarrow Y$ be a birational modification between compact K{\"a}hler manifolds and let $\ay\in H^{1,1}(Y,\R)$ be a big class. The volume is preserved by pull-backs, $$ \vol(f^{\star}\ay)=\vol(\ay) $$(see \cite{Bou1}). On the other hand, it is not preserved by push-forwards. In general we have $$\vol(f_\star \alpha_X)\geq \vol(\alpha_X)$$
(see Remark \ref{volpush}).

\subsection{Full mass currents}
Fix $X$ a $n$-dimensional compact K{\"a}hler manifold, $\a\in H^{1,1}(X,\R)$ be a big class and $\theta\in \a$ a smooth representative.
\subsubsection{The non-pluripolar product}
Let $T$ be a closed positive $(1,1)$-current. We denote by $\langle T^n\rangle$ the \emph{ non-pluripolar product} of $T$ defined in \cite[Proposition 1.6]{begz}.\\
Let us stress that since the non-pluripolar product does not charge pluripolar sets,
\begin{equation}\label{volTmin}
\vol(\a)=\int_X\langle T_{\min}^n\rangle
\end{equation}
whereas by \cite[Proposition 1.20]{begz} for any positive $(1,1)$-current $T\in \alpha$ we have 
\begin{equation}\label{volT}
\vol(\a)\geq \int_X\langle T^n\rangle.
\end{equation}
\begin{defi}\label{defi:fullMA} 
{\it A closed positive $(1,1)$-current $T$ on $X$ with cohomology class $\a$ is said to have \emph{full Monge-Amp{\`e}re mass} if
$$\int_X\langle T^n\rangle=\vol(\a).$$
We denote by $\cE(X,\a)$ the set of such currents.} 
{\it Let $\f$ be a $\theta$-psh function such that $T=\theta+dd^c\f$. 
The \emph{non-pluripolar Monge-Amp{\`e}re measure} of $\f$ is
$$
\MA(\f):=\langle(\theta+dd^c\f)^n\rangle=\langle T^n\rangle.
$$
We will say that $\f$ has \emph{full Monge-Amp{\`e}re mass} if $\theta+dd^c\f$ has full Monge-Amp{\`e}re mass. We denote by $\cE(X,\theta)$ the set of corresponding functions.}
\end{defi}

\section{Proof of the Main Theorem}
Throughout this section $X$ and $Y$ will be compact K\"ahler manifolds of complex dimension $n$. 

\begin{thm}\label{divisorialZariski}
Let $\alpha$ be a big class on $X$. Let $\alpha_1\in H^{1,1}(X,\R)$ be a modified nef class and $\alpha_2\in H^{1,1}(X,\R)$ be a pseudoeffective class. Then $\alpha=\alpha_1+\alpha_2$ is the divisorial Zariski decomposition of $\alpha$ if and only if $\vol(\alpha)=\vol(\alpha_1)$.
\end{thm}

\begin{rem}
In particular, Theorem \ref{divisorialZariski} implies that the pseudoeffective class $\alpha_2$ will be of the form $\alpha_2=\sum_{j=1}^N \lambda_j \{D_j\}$ where $D_j$ are prime divisors and $\lambda_j=\nu(\alpha,D_j)\geq 0$.
\end{rem}


\begin{proof}[Proof of Theorem \ref{divisorialZariski}]
If $\alpha=\alpha_1+\alpha_2$ is the divisorial Zariski decomposition then by \cite[Proposition 3.20]{Bou2} we have $	\vol(\alpha)=\vol(\alpha_1)$.\\
Viceversa, assume that we have a decomposition as above with $\vol(\alpha)=\vol(\alpha_1)=V$. Let $\mu$ be a smooth volume form on $X$ with total mass $V$ and let $T_1\in\cE(X,\alpha_1)$ be the unique solution of the complex Monge-Amp\`ere equation
$$\langle T_1^n\rangle=\mu.$$ Such $T_1$ exists and is unique by \cite[Theorem 3.1]{begz}. Furtheremore, $T_1$ has minimal singularities in its cohomology class \cite[Theorem 4.1]{begz}. Let $\tau$ be any positive closed $(1,1)$-current in $\alpha_2$ and set $T=T_1+\tau$.  By multilinearity of the non-pluripolar product \cite[Proposition 1.4]{begz}, we have $\langle T^n \rangle \geq \langle T_1^n\rangle $. By (\ref{volTmin}) and (\ref{volT}) we have $$\int_X  \langle T^n \rangle \leq \vol(\alpha)=\vol(\alpha_1)=\int_X \langle T_1^n \rangle .$$
Therefore $\langle T^n \rangle =\langle T_1^n \rangle =\mu$. \\
Thus $T$ is a solution of the Monge-Amp\`ere equation $\langle T^n \rangle =\mu$ in the class $\alpha$ and by uniqueness, it follows that $\alpha_2$ is rigid, i.e. there exists a unique positive closed $(1,1)$-current in $\alpha_2$. Moreover, $T$ has minimal singularities, so $\vol(\alpha)=\int_X \langle T^n \rangle$. Then by the multilinearity of the non-pluripolar product
$$\sum_{j=0}^{n-1} \binom{n}{j} \langle T_1^j \wedge \tau ^{n-j}\rangle=0.$$
Let $S\in\alpha_1$ be a K\"ahler current, i.e. $S \geq \varepsilon \omega$ for some $\varepsilon>0$. Let $\Omega_1$ be a non-empty Zariski open subset where $S$ is smooth and let $\Omega=\Amp(\alpha)\neq \emptyset$. Since $T$ has minimal singularities, then $T\in \alpha$  has locally bounded potential on $\Omega$. In particular, the current $\tau $ has locally bounded potential in $\Omega_2=\Omega \cap \Omega_1=X\setminus \Sigma$. Then we have
$$0\leq \varepsilon^{n-1}\int_{\Omega_2} \omega^{n-1} \wedge \tau \leq \int_{\Omega_2} S^{n-1} \wedge \tau\leq \int_{\Omega_2} T_1^{n-1} \wedge \tau=0$$
where the last inequality follows from \cite[Proposition 1.20]{begz}. This implies that the current $\tau$ is supported on $\Sigma$.\\
By \cite[Corollary 2.14]{Dem}, $\tau$ is of the form $$\tau=\sum_{j=1}^N \lambda_j [D_j]$$ where $D_j$ are irreducibile divisors and $\lambda_j\geq 0$. Moreover, observe that, since $\alpha_1$ is modified nef and $T_1 $ has minimal singularities we have $\nu(T_1, D_j)=0$ for any $j$ by Proposition \ref{bou2} hence $\lambda_j=\nu(T, D_j)$. In other words, $T=T_1+\tau$ is the Siu decomposition of $T$. Since $\alpha$ is big and $T$ has minimal singularities, by Proposition \ref{bou6ii} we have $\nu(\alpha, D)=\nu(T, D)$, hence the conclusion.
\end{proof}
\noindent We would like to observe that in the algebraic case, for a projective manifold $X$, Theorem \ref{divisorialZariski} can be proved using the differentiability of the volume \cite{BFJ}. \\
We thank S\'ebastien Boucksom for the following remark:
\begin{rem}\label{zariskiproj}
Let $N^1(X)_{\R}\subset H^{1,1}(X,\R)$ denote the real N\'eron-Severi space and $\alpha\in N^1(X)_{\R}$ be a big class. Assume $\alpha=\alpha_1+ \sum_{i=1}^N \lambda_i \{D_i\}$ with 
\begin{itemize}
\item[(i)] $\alpha_1\in N^1(X)_{\R}$ a modified nef class such that $\vol(\alpha)=\vol(\alpha_1)$;
\item[(ii)] $\lambda_i\geq 0$;
\item[(iii)] $D_i$ are prime divisors for  any $i$.
\end{itemize}
Then $\alpha=\alpha_1+ \sum_{i=1}^N \lambda_i \{D_i\}$ is the divisorial Zariski decomposition of $\alpha.$\\
We claim that it is enough to prove that for any prime divisor $D \not\subset E_{nk}(\alpha)$,
\begin{equation}\label{volIncreasing}
\vol(\alpha_1+tD) >\vol(\alpha_1) \,\quad \forall t>0.
\end{equation}
Indeed, to prove that $\alpha=\alpha_1+ \sum_{i=1}^N \lambda_i \{D_i\}$ is the divisorial Zariski decomposition of $\alpha$, we have to check that $D_i \subset E_{nk}(\alpha_1)$ by Property \ref{properties}(5). If $\lambda_i>0$ and $D_i\not\subset E_{nk}(\alpha_1)$ then (\ref{volIncreasing}) yields
$$\vol(\alpha)\geq \vol(\alpha_1+ tD_i) >\vol(\alpha_1)=\vol(\alpha), $$
hence a contradiction.\\
The inequality (\ref{volIncreasing}) easily follows from the differentiability of the volume. Indeed, by \cite[Theorem A]{BFJ} we have $$\frac{d}{dt}\big |_{t=0} \vol(\alpha_1+tD)=n \langle \alpha_1^{n-1}\rangle \cdot D$$
where $\langle \alpha_1^{n-1}\rangle$ denotes \emph{the positive product} of $\alpha$ defined in \cite[Definition 1.17]{begz}. Thanks to \cite[Remark 4.2 and Theorem 4.9]{BFJ}, we have $\langle \alpha_1^{n-1}\rangle \cdot D>0$, hence $\vol(\alpha_1+tD)$ is a continuous strictly increasing function for small $t>0$, and so $\vol(\alpha_1+tD)>\vol(\alpha_1)$.
\end{rem}

\section{Currents with full Monge-Amp\`ere mass}\label{fullMassCurrents}
In this section we state a few consequences of Theorem \ref{divisorialZariski}. The first result states that currents with full Monge-Amp\`ere mass in $\alpha$ compute the coefficients of the divisorial Zariski decomposition of $\alpha$.
\begin{thm}\label{lelongNumber}
Let $\alpha$ be a big class on $X$. If $T\in \cE(X,\alpha)$ and $T_{\min}\in\alpha$ is a current with mininal singularities, then the set
$$\{x\in X \,:\, \nu(T,x)> \nu(T_{\min},x)\}$$ is contained in a countable union of analytic subsets of codimension $\geq 2$ contained in $E_{nK}(\alpha)$. In particular, $\nu(T,D)=\nu(T_{\min},D)$ for any irreducible divisor $D\subset X$.
\end{thm}

\begin{proof}
If $T\in \cE(X,\alpha)$ then $E_+(T)\subset E_{nk}(\alpha)$ because of \cite[Proposition 1.9]{dn}. On the other hand if we write the Siu decomposition of $T$ as
$$T=T_1+\sum_{j\geq 1} \lambda_j [D_j]$$
where $D_j$ are prime divisors and ${\rm codim}\, E_c(T_1)\geq 2$ for all $c>0$, we have $D_j  \subset X\setminus \Amp(\alpha)$. Hence there is a finite number of $D_j$ such that $\lambda_j\neq 0$. In particular, $\nu(T_1,D_j)=0$ for any $j$.\\
Set $\alpha_1:=\{T_1\}$ and note that, since $\alpha$ is big, $\alpha_1$ is big. Moreover, $\alpha_1$ is modified nef. Indeed, pick $T_{\min,1}\in \alpha_1$ a current with minimal singularities. Since $0\leq \nu(T_{\min,1},D_j)\leq \nu(T_1, D_j)=0$, we have $\nu(T_{\min,1},D)=0$ for any $D$ prime divisor. The claim then follows from Propositions \ref{bou2} and \ref{bou6ii}.\\
Furthermore, the current $S=T_{\min,1}+\sum_{j=1}^N \lambda_j[D_j]$ is less singular than $T$, hence with full Monge-Amp\`ere mass \cite[Corollary 2.3]{begz}. Therefore
$$\vol(\alpha)=\int_X \langle T^n\rangle= \int_X \langle S^n\rangle=\int_X \langle T_{\min,1}^n\rangle=\vol(\alpha_1).$$
We are now under the assumptions of Theorem \ref{divisorialZariski}, thus $\alpha=\alpha_1+\sum_{j\geq 1} \lambda_j [D_j] $ is the divisorial Zariski decompostion of $\alpha$ and
 $$\nu(T,D_j)=\lambda_j=\nu(\alpha, D_j)=\nu(T_{\min}, D_j)$$
where the last identity is Proposition \ref{bou6ii}.\\
Moreover, $$B:=\{x\in X\;:\; \nu(T,x)> \nu(T_{\min}, x)\} \subset \bigcup_{c\in \Q^+} E_c (T_1) \cup \bigcup_{j=1}^N \Sigma_j,$$
where $\Sigma_j:=\{x\in D_j \; :\; \nu(T,x)>\lambda_j\}$. Indeed, if $x\in B$ is such that $x\in X\setminus \bigcup_{j=1}^N D_j$ then $\nu(T, x)= \nu(T_1,x)> \nu(T_{\min,1},x)\geq 0$. If $x\in D_j$ for some $j$ and $x\in B$ then $\nu(T,x)>\nu(T_{\min},D_j)=\lambda_j$, that is $x\in \Sigma_j$. Finally, observe that by \cite{Siu} both $E_c(T_1)$ and $\Sigma_j$ are analytic subsets of codimension $\geq 2$ for any $c>0$ and $j$, respectively.
\end{proof}

In \cite{dn}, the first named author proved that finite energy classes (and in particular the energy class $\cE$ defined in section \ref{fullMassCurrents}) are in ge\-neral not preserved by bimeromorphic maps (see \cite[Example 1.7 and Proposition 2.3]{dn}). In order to circumvent this problem she introduced a natural condition.

\begin{defi}\label{V}
\it{Let $f: X \dasharrow Y$ be a bimeromorphic map and $\a$ be a big class on $X$. Let $\mathcal{T}_{\a}(X)$ denote the set of positive closed $(1,1)$-currents in $\a$. We say that Condition (\texttt{V}) is satisfied if 
\begin{equation*}
f_{\star} \Big( \mathcal{T}_{\a}(X)  \Big)= \mathcal{T}_{f_{\star} \a}(Y)
\end{equation*}
where $\mathcal{T}_{f_{\star} \a}(Y)$ is the set of positive currents in the image class $f_{\star} \a$.}
\end{defi}
\begin{rem}\label{volpush}
Note that in general we have $f_{\star} \Big( \mathcal{T}_{\a}(X)  \Big)\subseteq  \mathcal{T}_{f_{\star} \a}(Y)$. This means in particular that the push-forward of a current with minimal singularities in $\a_X$ has not necessarly minimal singularities in $f_\star \a_X$, hence $\vol(f_\star \alpha_X)\geq \vol( \alpha_X)$.
\end{rem}
\noindent The first named author showed \cite[Proposition 2.3]{dn} that Condition (\texttt{V}) implies that $f_\star \cE(X, \alpha)= \cE(Y, f_\star \alpha)$.
\vspace{2mm}


\vspace{2mm}
In the following we prove that Condition (\texttt{V}) is equivalent to the preservation of volumes.

\begin{lem}\label{equiv}
Let $f: X \rightarrow Y$ be a birational morphism and let $\a$ be a big class on $X$. Let $E_i, F_i$ be distinct prime divisors contained in the exceptional locus $Exc(f)$ of $f$, then there exist $a_i, b_i\in \R^+$ such that
\begin{equation}\label{supp}
\alpha= f^\star f_\star \alpha-\left [\sum_i a_i\{E_i\}- \sum_i b_i\{F_i\}\right].
\end{equation}
Moreover, Condition (\texttt{V}) is equivalent to 
\begin{itemize}
\item [(i)]$a_i\leq \nu(f^\star f_\star\alpha, E_i)$ for any $i$;
\item[(ii)]$-b_i\leq \nu(f^\star f_\star\alpha, F_i)$ for any $i$.
\end{itemize}
\end{lem}
\noindent The statements in Lemma \ref{equiv} are quite standard but we include a proof for the reader's convenience.
\begin{proof}
The identity (\ref{supp}) follows from the fact that for any $T\in \alpha$ positive $(1,1)$-current, $T-f^\star f_\star T$ is supported on $Exc(f)$ since  $f$ is a biholomorphism on $X\setminus Exc(f)$. Therefore we conclude by \cite[Corollary 2.14]{Dem}.\\

\noindent Assume Condition (\texttt{V}) holds, that is, that any positive $(1,1)$-current $S\in f_\star \alpha$ can be written as $S=f_\star T$ for some positive $(1,1)$-current $T\in \alpha$. Since the cohomology classes of the excetional divisors of $f$ are linearly independent, by (\ref{supp}) we have an equality of currents
$$T+ \sum_i a_i [E_i]= f^\star f_\star T + \sum_i b_i [F_i].$$
Thus, for any $i$ we have $\nu(f^\star f_\star T, E_i)-a_i\geq 0$ and $\nu(f^\star f_\star T, F_i)+b_i\geq 0$. Hence (i) and (ii) since Condition (\texttt{V}) holds in particular for currents with minimal singularities in $f_\star \alpha$.\\
Conversely, let $S\in f_\star \alpha$ be a positive $(1,1)$-current. By the Siu decomposition the current
$$f^\star S -\sum_i \nu(f^\star S, E_i)[E_i]- \sum_i \nu(f^\star S, F_i)[F_i]$$ is positive. For any $i$, set $\lambda_i:= \nu(f^\star S, E_i)-a_i$ and $\mu_i:= \nu(f^\star S, F_i)+b_i$ and observe $\lambda_i,\mu_i\geq 0$ by (i) and (ii). Then
$$T:= f^\star S -\sum_i \nu(f^\star S, E_i)[E_i]- \sum_i \nu(f^\star S, F_i)[F_i]+ \sum_i\lambda_i [E_i]+\sum_i\mu_i [F_i]$$
is a positive $(1,1)$-current in $\alpha$ and by construction we have $f_\star T=S$. \\
\end{proof}

\begin{thm}\label{ConditionV}
Let $f: X \dasharrow Y$ be a bimeromorphic map and let $\a$ be a big class on $X$. Then Condition (\texttt{V}) holds if and only if $\vol(\a)=\vol(f_{\star}\a)$.
\end{thm}


\begin{proof}
Condition (\texttt{V}) insures that there exists a positive current $T\in \alpha$ such that $f_\star T$ is a current with minimal singularities in $f_\star \alpha$. Then
$$\vol(\alpha)\geq \int_X \langle T^n \rangle = \int_Y \langle (f_\star T) ^n \rangle=\vol(f_\star \alpha).$$ 
By Remark \ref{volpush} we get $\vol(\a)=\vol(f_{\star} \a)$. \\
Let us now prove the converse implication.\\
First, observe that, applying a resolution of singularities, a bimeromorphic map $f:X\dasharrow Y$ can be decomposed as $f=h^{-1}\circ g$,
\begin{displaymath}
\xymatrix{ & Z  \ar[dl]_{h} \ar[dr]^{g} 
 \\
 X &  &  Y }
\end{displaymath}
where $h,g$ are two birational morphisms and $Z$ denotes a resolution of singularities for the graph of $f$. By the proof of \cite[Proposition 1.12]{begz}, for every birational morphism $h$ we have $h^{\star} \Big( \mathcal{T}_{\a}(X)  \Big)= \mathcal{T}_{h^{\star} \a}(Z)$, hence it suffices to prove the claim when $f$ is a birational morphism.\\

\noindent Let $E_i, F_i$ and $a_i, b_i$ as in (\ref{supp}). By Lemma \ref{equiv}, Condition (\texttt{V}) is equivalent to:
\begin{itemize}
\item [(i)]$a_i\leq \nu(f^\star f_\star\alpha, E_i)$ for any $i$;
\item[(ii)]$-b_i\leq \nu(f^\star f_\star\alpha, F_i)$ for any $i$ .
\end{itemize}
\noindent Condition (ii) is satisfied since $\nu(f^\star f_\star\alpha, F_i)\geq 0$. Thus we are left proving (i).\\
Consider $\beta:=f^\star f_\star \alpha+ \sum_i b_i\{F_i\}$. We notice that $f_\star \beta=f_\star \alpha$. Moreover, by Lemma \ref{equiv}, $\beta$ satisfies Condition (\texttt{V}). Indeed, for any $i$ we have $-b_i\leq \nu (f^\star f_\star \beta, F_i)=\nu (f^\star f_\star \alpha, F_i)$. By the first implication of this theorem, we get $\vol(\beta)=\vol(f_\star \beta)=\vol(f_\star \alpha)$.\\
Let $T_{\min}\in \alpha$ and $S_{\min}\in f_\star \alpha$ be currents with minimal singularities. Then $T_{\min}+\sum_i a_i[E_i]$ and $f^\star S_{\min} +\sum_i b_i[F_i]$ are both positive $(1,1)$-currents in $\beta $ with full Monge-Amp\`ere mass. Indeed,
$$\int_X \langle  \left(T_{\min}+\sum_i a_i[E_i]\right)^n\rangle = \int_X \langle T_{\min}^n\rangle=\vol(\alpha)$$
$$\int_X \langle  \left(f^\star S_{\min} +\sum_i b_i[F_i]\right)^n \rangle=\int_Y \langle  S_{\min}^n\rangle=\vol(f_\star \alpha) $$
and $\vol(\alpha)=\vol(f_\star \alpha)= \vol(\beta)$. \\
By Theorem \ref{lelongNumber} $$a_j\leq \nu(T_{\min}+\sum_i a_i[E_i], \, E_j)=\nu(f^\star S_{\min} +\sum_i b_i[F_i],\, E_j)= \nu(f^\star S_{\min}, E_j)$$
for any prime divisor $E_j$, since the prime divisors $F_i$ and $E_j$ are distinct. By Proposition \ref{bou6ii}, $a_j\leq \nu(f^\star S_{\min},\, E_j)=\nu(f^\star f_\star \alpha,\, E_j)$, hence the conclusion.
\end{proof}

\begin{thm}
Let $\alpha$ be a big class and $D$ be an irreducible divisor such that $D\cap \Amp(\alpha)\neq \emptyset$. Then
$$\vol(\alpha+tD)>\vol(\alpha)\quad \forall t>0.$$
Viceversa, if $D\cap \Amp(\alpha)= \emptyset$ then
$$\vol(\alpha+tD)=\vol(\alpha)\quad \forall t>0.$$
\end{thm}

\begin{proof}
We first reduce to the case $\alpha$ modified nef and big class. Let $\alpha= Z(\alpha)+\{N(\alpha)\}$ be the divisorial Zariski decomposition of $\alpha$. By Lemma \ref{nonK} $D\cap \Amp(\alpha)\neq \emptyset$ if and only if $D \cap \Amp(Z(\alpha))\neq \emptyset$. \\
If the theorem is true for modified nef and big classes, we have
$$\vol(\alpha+tD)\geq \vol(Z(\alpha)+tD)>\vol(Z(\alpha))=\vol(\alpha).$$
Thus we can assume that $\alpha$ is a modified nef and big class. Assume by contradiction that there exists $t_0$ such that $\vol(\alpha+t_0 D)=\vol(\alpha)$. It follows by Theorem \ref{divisorialZariski} that $\beta=\alpha+t_0 D$ is the divisorial Zariski decomposition of $\beta$ and so $D\subset E_{nk}(\alpha)$ Property \ref{properties}(5). Since $E_{nk}(\alpha)=X\setminus \Amp(\alpha)$ \cite[Proposition 3.17]{Bou2} we get a contradiction.
\vspace{2mm}

Viceversa, if $\alpha=Z(\alpha)+\{N(\alpha)\}$ is the divisorial Zariski decomposition of $\alpha$ and $D\cap\Amp(\alpha)=\emptyset$ (or equivalently $D\subset E_{nk}(Z(\alpha))$ by Lemma \ref{nonK} below and \cite[Theorem 3.17]{Bou2}) then by Property \ref{properties}(5) we have that, for any $t>0$, the divisorial Zariski decomposition of $\alpha+tD$ is
$$\alpha+tD=Z(\alpha)+(N(\alpha)+tD)$$ thus $\vol(\alpha+tD)=\vol(Z(\alpha))=\vol(\alpha)$.
\end{proof}

\begin{lem}\label{nonK}
Let $\alpha\in H^{1,1}_{big}(X,\R)$ and let $\alpha=Z(\alpha)+\{N(\alpha)\}$ be its divisorial Zariski decomposition. Then we have
$$\Amp(\alpha)=\Amp(Z(\alpha))$$
\end{lem}

\begin{proof}
We first show the inclusion $\Amp(\alpha)\subset \Amp(Z(\alpha))$. Pick $x\in\Amp(\alpha)$. By definition there exists a K\"ahler current with analytic singularities $T\in\alpha$ which is smooth in a neighbourhood of $x$. Moreover $\nu(T_{\min}, x)=0$ since $0=\nu(T,x)\geq \nu(T_{\min},x)$. Let $T=R+\sum_j a_j [D_j]$ be the Siu decomposition of $T$, then $x\notin \supp D_j$ for any $j$. The current $T-N(\alpha)\in Z(\alpha)$ has clearly analytic singularities, is smooth around $x$ and it is also K\"ahler since $N(\alpha)\leq \sum_j a_j [D_j]$ by Proposition \ref{bou6ii}. Hence $x\in \Amp(Z(\alpha))$. 
Conversely, pick $x\in \Amp(Z(\alpha))$, then there exists a K\"ahler current with analytic singularities $T\in Z(\alpha)$ that is smooth in a neighbourhood of $x$ (see Definition \ref{ample}). It follows from Property \ref{properties}(5) that $x\notin \supp N(\alpha)$. This implies that $T+N(\alpha)\in \alpha$ is a K\"ahler current with analytic singularites that is smooth in a  neighbourhood of $x$. Hence $x\in \Amp(\alpha)$. 
\end{proof}

\end{document}